\documentclass[12pt]{amsart}
\usepackage{bbm}  

\usepackage[numbers,sort&compress]{natbib}  

\usepackage[top=1in, bottom=1in, left=1in, right=1in]{geometry}
\usepackage{times}

\usepackage{amssymb,amsmath,amsthm}
\usepackage{mathrsfs} 
\usepackage{enumitem}
\usepackage[hyphens]{url}
\usepackage{dsfont}
\usepackage{graphicx,tikz}

\usepackage[colorlinks=true, linkcolor=blue, citecolor=blue, urlcolor=purple]{hyperref}   

\usepackage{graphicx}
\usepackage{color}

\numberwithin{equation}{section}

\theoremstyle{plain}
\newtheorem{thm}{Theorem}
\newtheorem{lem}{Lemma}

\newtheorem{cor}{Corollary}

\theoremstyle{remark}
\newtheorem{remark}{Remark}

\theoremstyle{definition}

\theoremstyle{remark}

\newcommand{\bs}\boldsymbol{}

\renewcommand{\phi}{\varphi}

\renewcommand{\mod}[1]{\,({\rm mod}\,#1)}

\definecolor{red}{rgb}{1,0,0}
\definecolor{orange}{rgb}{0.7,0.3,0}
\definecolor{blue}{rgb}{.2,.6,.75}
\definecolor{green}{rgb}{.4,.7,.4}

\begin{document}
\title[Some Omega results for Dirichlet $L$-functions]
        {Some Omega results for Dirichlet $L$-functions}

\author{Qiyu Yang$^{1}$}

\author{Shengbo Zhao$^{2,\dag}$}

\thanks
{
$\dag$ Corresponding author.\\
$^{1}$ School of Mathematics and Statistics, Henan Normal University, Xinxiang, 453007, People's Republic of China, E-mail address: qyyang.must@gmail.com.\\
$^{2}$ School of Mathematical Sciences,
Key Laboratory of Intelligent Computing and Applications(Ministry of Education),
Tongji University, Shanghai 200092, China, E-mail address: shengbozhao@hotmail.com.
}

\keywords{Large value, Dirichlet $L$-functions, Resonance method}

\subjclass[2010]{11M06,  11M26}

\date{\today}

\begin{abstract}
Motivated by the first author's earlier work in 2024, we use the resonance method to establish some Omega results for Dirichlet $L$-functions, extending the previous results.
\end{abstract}


\maketitle

\section{Introduction} \label{sec:intr}
The study of large values of Dirichlet $L$-functions is a central topic in analytic number theory. The importance of these problems has been understood since the work of Littlewood, owing to their close connections with classical problems such as character sums, quadratic non-residues, and class numbers.

The Riemann zeta function is the simplest example of Dirichlet $L$-functions, and estimates for its order are closely related to the distribution of prime numbers. Using a method based on Diophantine approximation, Montgomery \cite{Montgomery1977} proved that
for any fixed $1/2<\sigma<1$,
\begin{equation}\label{Montgomery-bound}
\begin{aligned}
\underset{T^{(\sigma-1/2)/3}\leq t\leq T}{\max}\log|\zeta(\sigma+it)|
\geq\frac{1}{20}(\sigma-1/2)^{1/2}\frac{(\log T)^{1-\sigma}}{(\log\log T)^{\sigma}}.
\end{aligned}
\end{equation}
Regarding the upper bound of $|\log \zeta(\sigma+it)|$, Montgomery \cite{Montgomery1977} conjectured that
\[
|\log\zeta(\sigma+it)|\ll_{\sigma}\frac{(\log t)^{1-\sigma}}{(\log\log t)^{\sigma}}
\]
for $1/2<\sigma<1$. Therefore, the lower bound in \eqref{Montgomery-bound} should be optimal up to the constant. For further discussion of this problem, see \cite{Aist2016,Bondarenko2018,de la2019,Dong2022,Jaskari2024,Sound2008,Sound2021,Yangdaodao2023,Yangdaodao2024,Yang2024,Yang2026}.

In \cite{Aist2019}, Aistleitner {\it et al.} proved that, for all sufficiently large $q$, there exists a non-principal character $\chi \pmod q$ and a constant $C(\sigma)$ such that
\begin{equation}\label{Aist2019}
\begin{aligned}
\log|L(\sigma,\chi)|\geq C(\sigma)\frac{(\log q)^{1-\sigma}}{(\log\log q)^{\sigma}}
\end{aligned}
\end{equation}
for any fixed $1/2<\sigma<1$. Later, the first author \cite{Xiao2022} refined the constant
$C(\sigma)$ mentioned above.

Recently, Daodao Yang \cite{Yangdaodao2023} also proved that for any fixed $1/2<\sigma<1$, then
\begin{equation}\label{DaodaoYang-lower-bound}
\begin{aligned}
\max_{\substack{\chi\neq\chi_0\\ \chi\mod q}}-\Re\frac{L'}{L}(\sigma,\chi)
\geq c(\sigma)(\log q)^{1-\sigma}(\log\log q)^{1-\sigma},
\end{aligned}
\end{equation}
holds for all sufficiently large prime $q$, where $c(\sigma)$ is some positive constant which can be effectively computed.

\begin{remark}
In \cite{Yangdaodao2023}, Daodao Yang pointed out that $q$ in  \eqref{Aist2019} should be a prime number. Moreover, see the explanation on page 10 of \cite{Yangdaodao2023} and the counterexample concerning modulo $q$ given in Remark 3 of \cite{Yangdaodao2024}.
\end{remark}

 In this work, we will mainly consider large values of Dirichlet $L$-functions modulo a prime $q$. Specifically, we shall present the following results.
\begin{thm}\label{main-thm-1}
Let $\sigma\in(1/2,1)$ be fixed. For all sufficiently large prime $q$, we have
\begin{equation*}
\begin{aligned}
\max_{\substack{\chi\neq\chi_0\\ \chi\mod q}}
\Re\big(e^{-i\theta}\log L(\sigma,\chi)\big)\geq C_{\sigma,\theta}\frac{(\log q)^{1-\sigma}}{(\log\log q)^{\sigma}},
\end{aligned}
\end{equation*}
for any $\theta\in[0,\frac{\pi}{2}]\cup[\frac{3\pi}{2},2\pi]$. Here $C_{\sigma,\theta}$ is a positive number depending on $\sigma$ and $\theta$, satisfying
\[
C_{\sigma,\theta}<\cos\theta\bigg(\frac{\sigma}{1-\sigma}+o_{\sigma}(1)\bigg)
\bigg(\frac{1-\vartheta(\sigma-\epsilon)}{2\sigma}\bigg)^{1-\sigma}.
\]
\end{thm}

Applying the above with $\theta=0$ and $\theta=2\pi$, we obtain the following.

\begin{cor}\label{main-cor-1}
Let $\sigma\in(1/2,1)$ be fixed. For all sufficiently large prime $q$, we have
\begin{equation*}
\begin{aligned}
\max_{\substack{\chi\neq\chi_0\\ \chi\mod q}}
\log |L(\sigma,\chi)|\geq C_{\sigma}'\frac{(\log q)^{1-\sigma}}{(\log\log q)^{\sigma}}.
\end{aligned}
\end{equation*}
Here $C_{\sigma}'$ is a positive number depending on $\sigma$, satisfying
\[
C_{\sigma}'<\bigg(\frac{\sigma}{1-\sigma}+o_{\sigma}(1)\bigg)
\bigg(\frac{1-\vartheta(\sigma-\epsilon)}{2\sigma}\bigg)^{1-\sigma}.
\]
\end{cor}

\begin{remark}
By using a sharper zero-density estimate for $L(s,\chi)$ (see Lemma~\ref{zero-density-result}) and a more careful choice of parameters in the proof, we improve the constant appearing in \cite{Xiao2022}.
\end{remark}

\begin{thm}\label{main-thm-2}
Let $\sigma\in(1/2,1)$ be fixed. Assuming the GRH for $L(s,\chi)$, for all sufficiently large prime $q$, we have
\begin{equation*}
\begin{aligned}
\max_{\substack{\chi\neq\chi_0\\ \chi\mod q}}
\Re\big(e^{-i\theta}\log L(\sigma,\chi)\big)\geq D_{\sigma,\theta}\frac{(\log q)^{1-\sigma}}{(\log\log q)^{\sigma}},
\end{aligned}
\end{equation*}
for any $\theta\in[0,\frac{\pi}{2}]\cup[\frac{3\pi}{2},2\pi]$. Here $D_{\sigma,\theta}$ is a positive number depending on $\sigma$ and $\theta$, satisfying
\[
D_{\sigma,\theta}<\cos\theta\bigg(\frac{\sigma}{1-\sigma}+o_{\sigma}(1)\bigg)
\bigg(\frac{1}{2\sigma}\bigg)^{1-\sigma}.
\]
\end{thm}

Taking $\theta=0$ and $\theta=2\pi$ in the above argument, we arrive at the following result.

\begin{cor}\label{main-cor-2}
Let $\sigma\in(1/2,1)$ be fixed. Assuming the GRH for $L(s,\chi)$, for all sufficiently large prime $q$, we have
\begin{equation*}
\begin{aligned}
\max_{\substack{\chi\neq\chi_0\\ \chi\mod q}}
\log |L(\sigma,\chi)|\geq D_{\sigma}'\frac{(\log q)^{1-\sigma}}{(\log\log q)^{\sigma}}.
\end{aligned}
\end{equation*}
Here $D_{\sigma}'$ is a positive number depending on $\sigma$, satisfying
\[
D_{\sigma}'<\bigg(\frac{\sigma}{1-\sigma}+o_{\sigma}(1)\bigg)
\bigg(\frac{1}{2\sigma}\bigg)^{1-\sigma}.
\]
\end{cor}



\section{Preliminary lemmas}
Let $G_q$ denote the group of Dirichlet characters modulo $q$. We have the following lemma.
\begin{lem}\label{orthogonality}
If $(a,q)=1$, then
\begin{equation*}
\begin{aligned}
\frac{1}{\varphi(q)}\sum_{\chi\in G_q}\overline{\chi}(a)\chi(n)=\left\{
                                                           \begin{array}{ll}
                                                             1, & \hbox{if $n\equiv a \mod q, $} \\
                                                             0, & \hbox{otherwise.}
                                                           \end{array}
                                                         \right.
\end{aligned}
\end{equation*}
Here $\varphi(q)$ is the Euler's totient function.
\end{lem}
\begin{proof}
See \cite[Section 4, (4)]{Davenport}.
\end{proof}

Let $N(\sigma,T,\chi)$ denote the number of zeros of $L(s,\chi)$ in the rectangle $\{s\,:\,\Re(s)\geq\sigma,\,|\Im(s)|\leq T\}$. We have the following classical result.
\begin{lem}\label{zero-density-result}
Suppose that $q\geq1$ and $T\geq2$. We have
\begingroup
\renewcommand{\arraystretch}{4}
\begin{equation}\label{zero-density-result-1}
\begin{aligned}
\sum_{\chi\in G_q}N(\sigma,T,\chi)\ll
\begin{cases}
(qT)^{\frac{3(1-\sigma)}{2-\sigma}}(\log qT)^9, & \frac{1}{2} \leq \sigma \leq \frac{4}{5},\\
(qT)^{\frac{2(1-\sigma)}{\sigma}}(\log qT)^{14}, & \frac{4}{5} \leq \sigma \leq 1.
\end{cases}
\end{aligned}
\end{equation}
\endgroup
\end{lem}
\begin{proof}
See \cite[Theorem 12.1]{Montgomery1971}.
\end{proof}
\begin{remark}
For the sake of uniformity, we may rewrite \eqref{zero-density-result-1} in the following form:
\begin{equation}\label{zero-density-result-2}
\begin{aligned}
\sum_{\chi\in G_q}N(\sigma,T,\chi)\ll (qT)^{\vartheta(\sigma)+o_{q,T}(1)},
\end{aligned}
\end{equation}
where\footnote{Here, $o_{q,T}(1)\rightarrow0$ as $q,T\rightarrow\infty$.}
\[
\vartheta(\sigma):=\min\bigg(\frac{3(1-\sigma)}{2-\sigma},\frac{2(1-\sigma)}{\sigma}\bigg)
\quad\text{for}\quad 1/2\leq\sigma\leq1.
\]
\end{remark}

\begin{lem}\label{sound-Truncated}
Let $s=\sigma+it$ with $\sigma>1/2$ and $|t|\leq 2q$. Let $Y\geq2$ be a real number, and let $1/2\leq\sigma_0<\sigma$. Suppose that there are no zeros of $L(z,\chi)$ inside the rectangle $\{z:\sigma_0\leq\Re(z)\leq1,\,\,|\Im(z)-t|\leq Y+3\}$. Then
\begin{equation*}
\begin{aligned}
\log L(s,\chi)=\sum_{n=2}^{Y}\frac{\Lambda(n)\chi(n)}{n^{s}\log n}
+O\Big(\frac{\log q}{(\sigma_1-\sigma_0)^2}Y^{\sigma_1-\sigma}\Big),
\end{aligned}
\end{equation*}
where $\sigma_1=\min\big(\sigma_0+\frac{1}{\log Y},\frac{\sigma+\sigma_0}{2}\big)$.
\end{lem}

\begin{proof}
See \cite[Lemma 8.2]{Granville2001}.
\end{proof}

Set $Y=(\log q)^{\frac{100}{\sigma-1/2}}$. Define
\[
T_{\chi}(\sigma,Y):=\sum_{p\leq Y}\frac{\chi(p)}{p^{\sigma}}.
\]
We shall have the following truncated formula.
\begin{lem}\label{truncated-log-Dirichlet}
For any fixed $1/2<\sigma<1$, we have
\begin{equation*}
\begin{aligned}
e^{-i\theta}\log L(\sigma,\chi)=e^{-i\theta}T_{\chi}(\sigma,X)+O_{\sigma,Y}(1),\quad\forall\,\chi\in G_q\backslash \mathcal{E}_{q}.
\end{aligned}
\end{equation*}
Here, let $\mathcal{E}_{q}$ be a set of exceptional characters modulo $q$ with cardinality satisfying
\[
\#\mathcal{E}_{q}\ll q^{\vartheta(\sigma-\epsilon)+o_q(1)}.
\]
\end{lem}

\begin{proof}
For $1/2\leq\sigma_0\leq1$, it follows from \eqref{zero-density-result-2} that
\begin{equation}\label{Mont-zero-density}
\begin{aligned}
\sum_{\chi\in G_q}N(\sigma_0,T,\chi)\ll (qT)^{\vartheta(\sigma_0)+o_{q,T}(1)}.
\end{aligned}
\end{equation}

Let $\epsilon>0$ be a small number, and let $t=0$, $\sigma_0=\sigma-\epsilon$, $\sigma_1=\sigma_0+\frac{1}{\log Y}$, and $Y=(\log q)^{\frac{100}{\sigma-1/2}}$ in Lemma \ref{sound-Truncated}. Taking $T=Y+2$ and combining \eqref{Mont-zero-density}, we deduce that
\begin{equation*}
\begin{aligned}
\log L(\sigma,\chi)=\sum_{p\leq Y}\frac{\chi(p)}{p^{\sigma}}+O_{\sigma,Y}(1),
\end{aligned}
\end{equation*}
where $\mathcal{B}_{q}$ is a set of ``bad" characters with cardinality satisfying
\[
\#\mathcal{B}_{q}\ll q^{\vartheta(\sigma-\epsilon)+o_q(1)},
\]
and
\[
\vartheta(\sigma-\epsilon)=\min\bigg(\frac{3\big(1-(\sigma-\epsilon)\big)}{2-(\sigma-\epsilon)},\frac{2\big(1-(\sigma-\epsilon)\big)}{\sigma-\epsilon}\bigg)
\quad\text{for}\quad 1/2\leq\sigma\leq1.
\]

Set $\mathcal{E}_{q}=\mathcal{B}_{q}\cup\{\chi_0\}$, where $\chi_0$ is the trivial character modulo $q$. This completes the proof.
\end{proof}

\subsection{Constructing the resonator}
Let $X=A(\log q)(\log\log q)$ with the parameter $A$ still to be
chosen. As in \cite{Xiao2022}, we define
\[
r(p)=1-\Big(\frac{p}{X}\Big)^{\sigma}
\]
for $p\leq X$. Also set $r(1)=1$ and $r(p)=0$ for $p>X$. Moreover, we extend $r(n)$ to a completely multiplicative function. Next, we define
\begin{equation*}
\begin{aligned}
R(\chi):=\prod_{p\leq X}\big(1-r(p)\chi(p)\big)^{-1}
=\prod_{p\leq X}\sum_{\ell\geq0}\big(r(p)\chi(p)\big)^{\ell}=\sum_{n=1}^{\infty}r(n)\chi(n).
\end{aligned}
\end{equation*}
Note that
\begin{equation*}
\begin{aligned}
\log|R(\chi)|^2\leq 2\sigma\sum_{p\leq X}(\log X-\log p)
=\big(2\sigma+o_{\sigma}(1)\big)\frac{X}{\log X}
=\big(2\sigma A+o_{\sigma,q}(1)\big)\log q.
\end{aligned}
\end{equation*}
Thus we have
\begin{equation}\label{upper-bound-R-chi-2}
\begin{aligned}
|R(\chi)|^2\ll q^{2\sigma A+o_{\sigma,q}(1)}.
\end{aligned}
\end{equation}

Define
\begin{equation*}
\begin{aligned}
\mathcal{Q}_1(\chi,q)&=\sum_{\chi\in G_q}|R(\chi)|^2,\\
\mathcal{Q}_2(\sigma,Y,\chi,q)&=\sum_{\chi\in G_q}\Re\big(e^{-i\theta}T_{\chi}(\sigma,Y)\big)|R(\chi)|^2,\\
\mathcal{S}_1(\chi,q)&=\sum_{\chi\in G_q\setminus \mathcal{E}_q}|R(\chi)|^2,\\
\mathcal{S}_2(\sigma,Y,\chi,q)&=\sum_{\chi\in G_q\setminus \mathcal{E}_q}\Re\big(e^{-i\theta}T_{\chi}(\sigma,Y)\big)|R(\chi)|^2.
\end{aligned}
\end{equation*}

Set
\[
\lambda(\sigma):=\int_0^1\frac{t^{\sigma}}{2-t^{\sigma}}\,dt.
\]
We have the following lemma.
\begin{lem}\label{formula-logarithmic-Z1-Q1}
Let $A<\frac{1-\vartheta(\sigma-\epsilon)}{\sigma(1+\lambda(\sigma))}$. For all sufficiently large prime $q$, we have
\begin{equation}\label{formula-S1-Q1}
\begin{aligned}
\mathcal{S}_1(\chi,q)=\big(1+o(1)\big)\mathcal{Q}_1(\chi,q)
\end{aligned}
\end{equation}
and
\begin{equation}\label{formula-S2-Q2}
\begin{aligned}
\mathcal{S}_2(\sigma,Y,\chi,q)=\big(1+o(1)\big)\mathcal{Q}_2(\sigma,Y,\chi,q).
\end{aligned}
\end{equation}
\end{lem}

\begin{proof}
Note that
\[
R(\chi)=\prod_{p\leq X}\sum_{\ell\geq0}\big(r(p)\chi(p)\big)^{\ell}=\sum_{n=1}^{\infty}r(n)\chi(n).
\]
For all sufficiently large prime $q$, we deduce that
\begin{equation}\label{Q-1-chi-q}
\begin{aligned}
\mathcal{Q}_1(\chi,q)
=\varphi(q)\sum_{\substack{m,n=1 \\ m\equiv n\mod q \\ (n,q)=1}}^{\infty}r(m)r(n)
\geq(q-1)\sum_{n=1}^{\infty}r(n)^2
\gg q^{1+A\sigma(1-\lambda(\sigma))+o_q(1)}.
\end{aligned}
\end{equation}
The lower bound in \eqref{Q-1-chi-q} is obtained by adapting an argument due to Dong (see \cite[p.~79]{Dong2022}).

 By \eqref{upper-bound-R-chi-2}, it follows that
\begin{equation*}
\begin{aligned}
\sum_{\chi\in\mathcal{E}_q}|R(\chi)|^2
\ll q^{2\sigma A+o_{\sigma,q}(1)}\cdot q^{\vartheta(\sigma-\epsilon)+o_{q}(1)}
=o\big(q^{1+A\sigma(1-\lambda(\sigma))+o_q(1)}\big)
\end{aligned}
\end{equation*}
for $A<\frac{1-\vartheta(\sigma-\epsilon)}{\sigma(1+\lambda(\sigma))}$.

Combining the above formulae, we deduce that
\[
\mathcal{S}_1(\chi,q)=\bigg(\sum_{\chi\in G_q}-\sum_{\chi\in\mathcal{E}_q}\bigg)|R(\chi)|^2
=\mathcal{Q}_1(\chi,q)+o\big(q^{1+A\sigma(1-\lambda(\sigma))+o_q(1)}\big).
\]
Next, we shall give the proof of \eqref{formula-S2-Q2}. By Lemma \ref{orthogonality}, we have
\begin{equation}\label{lower-bound-Q2-1}
\begin{aligned}
\mathcal{Q}_2(\sigma,Y,\chi,q)&=\Re\bigg(\sum_{\chi\in G_q}e^{-i\theta}T_{\chi}(\sigma,Y)|R(\chi)|^2\bigg)\\
&=\Re\bigg(\sum_{\chi\in G_q}e^{-i\theta}\sum_{p\leq Y}\frac{\chi(p)}{p^{\sigma}}
\sum_{m,n=1}^{\infty}r(m)r(n)\chi(m)\overline{\chi(n)}\bigg)\\
&=\cos\theta\sum_{\substack{p\leq Y\\(p,q)=1}}\frac{1}{p^{\sigma}}
\bigg(\varphi(q)\sum_{\substack{m,n=1\\pm\equiv n\mod q\\(n,q)=1}}^{\infty}r(m)r(n)\bigg).
\end{aligned}
\end{equation}
Also, we observe that
\begin{equation}\label{lower-bound-Q2-2}
\begin{aligned}
\varphi(q)\sum_{\substack{m,n=1\\pm\equiv n\mod q\\(n,q)=1}}^{\infty}r(m)r(n)
\geq\varphi(q)\sum_{\substack{m,k=1\\pm\equiv pk\mod q\\(pk,q)=1}}^{\infty}r(m)r(pk)
= r(p)\mathcal{Q}_1(\chi,q).
\end{aligned}
\end{equation}
Noting that $r(p)=0$ for $p>X$ and combining \eqref{lower-bound-Q2-1} and \eqref{lower-bound-Q2-2}, we deduce that
\begin{equation}\label{ratio-lower-bound-Q2-Q1}
\begin{aligned}
\frac{\mathcal{Q}_2(\sigma,Y,\chi,q)}{\mathcal{Q}_1(\chi,q)}\geq\cos\theta\sum_{\substack{p\leq Y \\ (p,q)=1}}\frac{r(p)}{p^{\sigma}}
\geq\cos\theta\sum_{p\leq X}\frac{r(p)}{p^{\sigma}}
=\cos\theta\bigg(\frac{\sigma}{1-\sigma}+o_{\sigma}(1)\bigg)\frac{X^{1-\sigma}}{\log X},
\end{aligned}
\end{equation}
for any $\theta\in[0,\frac{\pi}{2}]\cup[\frac{3\pi}{2},2\pi]$.

Combining \eqref{Q-1-chi-q} and \eqref{ratio-lower-bound-Q2-Q1}, we arrive at
\begin{equation}\label{final-ratio-lower-bound-Q2-Q1}
\begin{aligned}
\mathcal{Q}_2(\sigma,Y,\chi,q)\gg q^{1+A\sigma(1-\lambda(\sigma))+o_q(1)}.
\end{aligned}
\end{equation}

According to
\[
e^{-i\theta}T_{\chi}(\sigma,Y)\ll\sum_{p\leq Y}\frac{1}{p^{\sigma}}\ll \frac{Y^{1-\sigma}}{\log Y}\ll q^{o_{\sigma,q}(1)},
\]
we have
\[
\sum_{\chi\in \mathcal{E}_q}e^{-i\theta}T_{\chi}(\sigma,Y)|R(\chi)|^2
\ll \frac{Y^{1-\sigma}}{\log Y}\cdot q^{2\sigma A+o_{\sigma,q}(1)}\cdot q^{\vartheta(\sigma-\epsilon)+o_{q}(1)}
=o\big(q^{1+A\sigma(1-\lambda(\sigma))+o_q(1)}\big)
\]
for $A<\frac{1-\vartheta(\sigma-\epsilon)}{\sigma(1+\lambda(\sigma))}$.

Hence
\[
\mathcal{S}_2(\sigma,Y,\chi,q)=\bigg(\sum_{\chi\in G_q}-\sum_{\chi\in\mathcal{E}_q}\bigg)e^{-i\theta}T_{\chi}(\sigma,Y)|R(\chi)|^2
=\mathcal{Q}_2(\sigma,Y,\chi,q)+o\big(q^{1+A\sigma(1-\lambda(\sigma))+o_q(1)}\big),
\]
where $\mathcal{Q}_2(\sigma,Y,\chi,q)\gg q^{1+A\sigma(1-\lambda(\sigma))+o_q(1)}$.

So the proof is complete.

\end{proof}

\section{Proof of Theorem \ref{main-thm-1}}
Combining \eqref{formula-S1-Q1}, \eqref{formula-S2-Q2}, and \eqref{final-ratio-lower-bound-Q2-Q1}, we obtain
\begin{equation}\label{final-proof-thm1}
\begin{aligned}
\max_{\substack{\chi\neq\chi_0\\ \chi\mod q}}\Re\big(e^{-i\theta}\log L(\sigma,\chi)\big)
&\geq
\frac
{
\sum_{\chi\in G_q\setminus \mathcal{E}_q}\Re\big(e^{-i\theta}\log L(\sigma,\chi)\big)|R(\chi)|^2
}
{
\sum_{\chi\in G_q\setminus \mathcal{E}_q}|R(\chi)|^2
}\\
&\geq\frac{\mathcal{Q}_2(\sigma,Y,\chi,q)}{\mathcal{Q}_1(\chi,q)}\\
&\geq\cos\theta\bigg(\frac{\sigma}{1-\sigma}+o_{\sigma}(1)\bigg)
A^{1-\sigma}\frac{(\log q)^{1-\sigma}}{(\log\log q)^{\sigma}},
\end{aligned}
\end{equation}
where $A<\frac{1-\vartheta(\sigma-\epsilon)}{\sigma(1+\lambda(\sigma))}$.

Finally, we complete the proof of Theorem \ref{main-thm-1}.

\section{Proof of Theorem \ref{main-thm-2}}
Let $\sigma\in(1/2,1)$ be fixed. Assuming the GRH for $L(s,\chi)$, we have
\begin{equation*}
\begin{aligned}
e^{-i\theta}\log L(\sigma,\chi)=e^{-i\theta}T_{\chi}(\sigma,X)+O_{\sigma,Y}(1),
\end{aligned}
\end{equation*}
for all character modulo $q$. Replacing Lemma \ref{truncated-log-Dirichlet} with the above formula,
Theorem \ref{main-thm-2} follows by an argument similar to that in the previous section.

\section{Concluding Remarks}
Following Daodao Yang's recent work in \cite[(32)]{Yangdaodao2023}, we obtain
\begin{equation}\label{formula-logarithmic-derivative}
\begin{aligned}
-\frac{L'}{L}(\sigma,\chi)=\sum_{n\leq Y}\frac{\Lambda(n)\chi(n)}{n^{\sigma}}+O((\log q)^{-8}),\quad\forall\,\chi\in G_q\setminus \mathscr{E}_q,
\end{aligned}
\end{equation}
where $\mathscr{E}_q$ is a set of of ``exceptional" characters.

Indeed, the terms when $n=p^{\ell}$ for $\ell\geq2$ contribute $O_{\sigma,Y}(1)$ in total
to $\sum_{n\leq Y}\frac{\Lambda(n)\chi(n)}{n^{\sigma}}$. Applying a sharper zero-density estimate for $L(s,\chi)$ (see Lemma \ref{zero-density-result}) together with the above formula \eqref{formula-logarithmic-derivative}, we have
\begin{equation}\label{formula-logarithmic-derivative}
\begin{aligned}
-e^{-i\theta}\frac{L'}{L}(\sigma,\chi)=e^{-i\theta}\sum_{p\leq Y}\frac{\log p}{p^{\sigma}}\chi(p)+O_{\sigma,Y}(1),\quad\forall\,\chi\in G_q\setminus \mathscr{E}_q,
\end{aligned}
\end{equation}
where $\mathscr{E}_{q}$ is a set of of ``bad" characters with cardinality satisfying
\[
\#\mathscr{E}_{q}\ll q^{\vartheta(\sigma-\epsilon)+o_q(1)},
\]
and
\[
\vartheta(\sigma-\epsilon):=\min\bigg(\frac{3\big(1-(\sigma-\epsilon)\big)}{2-(\sigma-\epsilon)},\frac{2\big(1-(\sigma-\epsilon)\big)}{\sigma-\epsilon}\bigg)
\quad\text{for}\quad 1/2\leq\sigma\leq1.
\]
Using a proof entirely similar to that of Theorem \ref{main-thm-1}, we obtain
\begin{equation*}
\begin{aligned}
\max_{\substack{\chi\neq\chi_0\\ \chi\mod q}}-\Re\bigg(e^{-i\theta}\frac{L'}{L}(\sigma,\chi)\bigg)
&\geq
\frac
{
-\sum_{\chi\in G_q\setminus \mathscr{E}_q}\Re\big(e^{-i\theta}\frac{L'}{L}(\sigma,\chi)\big)|R(\chi)|^2
}
{
\sum_{\chi\in G_q\setminus \mathscr{E}_q}|R(\chi)|^2
}\\
&\geq\cos\theta\bigg(\frac{\sigma}{1-\sigma}+o_{\sigma}(1)\bigg)
A^{1-\sigma}(\log q)^{1-\sigma}(\log\log q)^{1-\sigma},
\end{aligned}
\end{equation*}
for any $\theta\in[0,\frac{\pi}{2}]\cup[\frac{3\pi}{2},2\pi]$, where $A<\frac{1-\vartheta(\sigma-\epsilon)}{\sigma(1+\lambda(\sigma))}$.

In particular, we have
\begin{equation}\label{Re-log-der-lower}
\begin{aligned}
\max_{\substack{\chi\neq\chi_0\\ \chi\mod q}}-\Re\frac{L'}{L}(\sigma,\chi)
&\geq\bigg(\frac{\sigma}{1-\sigma}+o_{\sigma}(1)\bigg)
\bigg(\frac{1-\vartheta(\sigma-\epsilon)}{\sigma(1+\lambda(\sigma))}\bigg)^{1-\sigma}(\log q)^{1-\sigma}(\log\log q)^{1-\sigma}
\end{aligned}
\end{equation}
for all sufficiently large prime $q$. In fact, by applying a sharper zero-density result (see Lemma~\ref{zero-density-result}), Daodao Yang \cite{Yangdaodao2023} can obtain the same conclusion as in \eqref{Re-log-der-lower}.

\section*{Acknowledgments}
The authors would like to thank Dr. DaoDao Yang for helpful discussion. The first author was supported by the Natural Science Foundation of Henan Province (Grant No. 252300421782).

\bigskip

\end{document}